\documentclass[11pt,reqno]{amsart}
\usepackage{graphicx}
\usepackage{enumitem}
\setlist[itemize]{leftmargin=2.5em}
\setlist[enumerate]{font={\upshape}, label=\arabic*., leftmargin=2.5em}
\usepackage{mathtools}
\usepackage{amssymb}
\usepackage[sort&compress,numbers]{natbib} 

\usepackage[pdftex,
pagebackref=true
]{hyperref}
\hypersetup{
	colorlinks=true,
	allcolors=blue,
}
\usepackage[capitalize]{cleveref}
\usepackage[margin=0.8in]{geometry}
\newtheorem{lemma}{Lemma}[section]
\newtheorem{claim}[lemma]{Claim}
\Crefname{claim}{Claim}{Claims}
\newtheorem{theorem}[lemma]{Theorem}
\newtheorem{corollary}[lemma]{Corollary}
\newtheorem{conjecture}[lemma]{Conjecture}
\newtheorem{proposition}[lemma]{Proposition}
\expandafter\let\expandafter\oldproof\csname\string\proof\endcsname
\let\oldendproof\endproof
\renewenvironment{proof}[1][\proofname]{%
	\oldproof[\normalfont\bfseries #1]%
}{\oldendproof}
\newenvironment{subproof}[1][\normalfont\it\proofname]{%
	\begin{proof}[#1]%
	}{%
	\end{proof}%
}
\newcommand{\dd}{\textquotedblleft}
\newcommand{\ee}{\textquotedblright}

\newcommand{\vep}{\varepsilon}
\newcommand{\nin}{\notin}
\renewcommand{\subset}{\subseteq}
\renewcommand{\supset}{\supseteq}
\newcommand{\ind}{\operatorname{ind}}

\newcommand{\rdl}{R\"{o}dl}
\newcommand{\erh}{Erd\H{o}s--Hajnal}
\DeclarePairedDelimiter\abs{\lvert}{\rvert}%
\DeclarePairedDelimiter\ceil{\lceil}{\rceil}%
\DeclarePairedDelimiter\floor{\lfloor}{\rfloor}%
\makeatletter
\newcommand{\leqnomode}{\tagsleft@true}
\newcommand{\reqnomode}{\tagsleft@false}
\makeatother

\begin{document}	
	\title{A further extension of R\"{o}dl's theorem}
	\author{Tung H. Nguyen}
	\address{Princeton University, Princeton, NJ 08544, USA}
	\email{\href{mailto:tunghn@math.princeton.edu}
		{tunghn@math.princeton.edu}}
	\thanks{Partially supported by AFOSR grants A9550-10-1-0187 and FA9550-22-1-0234, and NSF grant DMS-2154169.}
	\begin{abstract}
		Fix $\vep>0$ and a graph $H$ with at least one vertex.
		A well-known theorem of \rdl{} from the 80s says that every graph $G$ with no induced copy of $H$ contains a linear-sized {\em $\vep$-restricted} set $S\subset V(G)$,
		which means $S$ induces a subgraph with maximum degree at most $\vep\abs{S}$ in $G$ or its complement.
		There are two extensions of this result:
		\begin{itemize}
			\item quantitatively, Nikiforov relaxed the condition \dd no induced copy of $H$\ee{} to \dd at most $\kappa\abs{G}^{\abs{H}}$ induced copies of $H$ for some $\kappa>0$ depending on $H$ and $\vep$;\ee{} and
			
			\item qualitatively, Chudnovsky, Scott, Seymour, and Spirkl recently showed that there exists $N>0$ depending on $H$ and $\vep$ such that $G$ is {\em $(N,\vep)$-restricted},
			which means $V(G)$ has a partition into at most $N$ subsets that are $\vep$-restricted.
		\end{itemize}
		
		A natural common generalization of these two asserts that every graph $G$ with at most $\kappa\abs{G}^{\abs{H}}$ induced copies of $H$ is $(N,\vep)$-restricted for some $\kappa,N>0$ depending on $H$ and $\vep$.
		This is unfortunately false;
		but we prove that for every $\vep>0$, $\kappa$ and $N$ still exist so that for every~$d\ge0$, every graph $G$ with at most $\kappa d^{\abs{H}}$ induced copies of $H$
		has an $(N,\vep)$-restricted induced subgraph on at least $\abs G-d$ vertices.
		This unifies the two aforementioned~theorems,
		and is optimal up to $\kappa$ and $N$ for every value of $d$.
	\end{abstract}
	\maketitle
	\section{Introduction}
	Graphs in this paper are finite and simple.
	For a graph $G$ with vertex set $V(G)$ and edge set $E(G)$,
	let $\abs G:=\abs{V(G)}$, and
	let $\overline G$ denote its complement.
	For $S\subset V(G)$,
	let $G[S]$ denote the subgraph of $G$ induced~by~$S$,
	and let $G\setminus S:=G[V(G)\setminus S]$.
	For a nonnull graph $H$,
	a \emph{copy} of $H$ in $G$ is a graph isomorphism from $H$ to $G[S]$ for some $S\subset V(G)$.
	Let $\ind_H(G)$ be the number of copies of $H$ in $G$;
	and say that $G$ is {\em $H$-free} if $\ind_H(G)=0$.
	Given $\vep>0$, a subset $S\subset V(G)$ is {\em $\vep$-restricted} in $G$ if one of $G[S],\overline{G}[S]$ has maximum degree at most $\vep\abs{S}$.
	The following well-known theorem of \rdl{}~\cite{rodl1986} from 1986 has become a standard tool in the investigation of the \erh{} conjecture\footnote{The very last sentence of~\cite{MR599767} was actually the first time Erd\H{o}s and Hajnal formally stated their well-known conjecture.}~\cite{MR1031262,MR599767}
	(see~\cite{MR3150572} for~a~survey).
	\begin{theorem}
		[\cite{rodl1986}]
		\label{thm:rodl}
		For every $\vep>0$ and every graph $H$, there exists $\delta=\delta(H,\vep)>0$ such that for every $H$-free graph $G$, there is an $\vep$-restricted $S\subset V(G)$ in $G$ with $\abs S\ge\delta\abs{G}$.
	\end{theorem}
	Since its inception, \cref{thm:rodl} has found many extensions.
	Among these is the following useful quantitative improvement first proved by Nikiforov~\cite{nikiforov2006}
	(see~\cite{MR4563865,nss2023a,nss2023b} for several applications).
	\begin{theorem}
		[\cite{nikiforov2006}]
		\label{thm:niki}
		For every $\vep>0$ and every graph $H$, there exist $\delta=\delta_{\ref{thm:niki}}(H,\vep)>0$
		and $\kappa=\kappa_{\ref{thm:niki}}(H,\vep)>0$ such that for every graph $G$ with $\ind_H(G)\le\kappa\abs{G}^{\abs H}$,
		there is an $\vep$-restricted $S\subset V(G)$ in $G$ with $\abs S\ge\delta\abs G$.
	\end{theorem}
	\rdl's original proof of \cref{thm:rodl} and Nikiforov's proof of \cref{thm:niki} (we remark that \cref{thm:niki} is already implicit in~\cite{rodl1986}) both employ the regularity lemma, and so give bounds on $\delta^{-1}$ and $\kappa^{-1}$ which are towers of twos of height polynomial in $\vep^{-1}$ with constants depending on $H$.
	Fox and Sudakov~\cite{fox2008} offered an alternative proof of \cref{thm:niki}
	showing that both $\delta$ and $\kappa$ can be chosen as $2^{-c\log^2(\vep^{-1})}$
	for some constant $c>0$ depending on $H$;
	and very recently Buci\'c, Nguyen, Scott, and Seymour~\cite{bnss2023} improved this to $2^{-c\log^2(\vep^{-1})/\log\log(\vep^{-1})}$.
	In~\cite{nss2023a}, it is conjectured that both $\delta$ and $\kappa$ can in fact be taken to be a polynomial of $\vep$ in \cref{thm:niki}, which would imply the \erh{} conjecture itself
	(see~\cite{fnss2023,nss2023a,nss2023b} for current progress on this~topic).
	
	Recently, Chudnovsky, Scott, Seymour, and Spirkl~\cite{chudnovsky2021-1} provided a qualitative refinement of \cref{thm:rodl},
	which says that the vertex set of every $H$-free graph can even be partitioned into a bounded number of $\vep$-restricted subsets.
	Formally, for $\vep,N>0$,
	a graph $G$ is {\em $(N,\vep)$-restricted} if there is a partition of $V(G)$ into at most $N$ subsets that are $\vep$-restricted in $G$;
	thus $G$ is $(N,\vep)$-restricted if and only if $\overline G$ is.
	\begin{theorem}
		[\cite{chudnovsky2021-1}]
		\label{thm:csss}
		For every $\vep>0$ and every graph $H$,
		there exists $N=N(H,\vep)>0$ such that every $H$-free graph is $(N,\vep)$-restricted.
	\end{theorem}
	The \emph{edge density} of a graph $G$ equals $\abs{E(G)}/{\abs G\choose2}$ if $\abs G\ge2$ and equals $0$ if $\abs G\le1$.
	For $\vep>0$,
	a subset $S\subset V(G)$ is {\em weakly $\vep$-restricted} in $G$ if one of $G[S],\overline{G}[S]$ has edge density at most $\vep$.
	Thus if $S$ is $\frac{1}{2}\vep$-restricted in $G$ then it is weakly $\vep$-restricted;
	and if $S$ is weakly $\frac{1}{4}\vep$-restricted in $G$ then it has an $\vep$-restricted subset of size $\ceil{\frac{1}{2}\abs S}$.
	Hence the strength of \cref{thm:rodl,thm:niki} remain unaffected if \dd$\vep$-restricted\ee{} is replaced by \dd weakly $\vep$-restricted.\ee{}
	As discussed in~\cite{chudnovsky2021-1}, however, \cref{thm:csss} becomes significantly weaker if \dd$(N,\vep)$-restricted\ee{} is replaced by \dd{\em weakly $(N,\vep)$-restricted},\ee{}
	which means $V(G)$  has a partition into at most $N$ subsets that are weakly $\vep$-restricted in $G$.
	Indeed, repeated applications of \cref{thm:niki} yield the~following result proved in~\cite{nikiforov2006}.
	\begin{theorem}
		[\cite{nikiforov2006}]
		\label{thm:weak}
		For every $\vep>0$ and every graph $H$, there exist $\kappa=\kappa(H,\vep)>0$ and $N=N(H,\vep)>0$ such that every graph $G$ with $\ind_H(G)\le\kappa\abs{G}^{\abs H}$
		is weakly $(N,\vep)$-restricted. 
	\end{theorem}
	(As shown in~\cite{nikiforov2006,fox2008}, with more care one can even take the corresponding weakly $\vep$-restricted sets to have size differences at most $1$ in this result.)
	It thus would be natural (and quite tempting) to conjecture the following,
	which would have unified \cref{thm:niki,thm:csss} and strengthened \cref{thm:weak} considerably.
	\begin{conjecture}[false]
		\label{conj:false}
		For every $\vep>0$ and every graph $H$, there exist $N=N(H,\vep)>0$ and $\kappa=\kappa(H,\vep)>0$ such that every graph $G$ with $\ind_H(G)\le\kappa\abs{G}^{\abs H}$ is $(N,\vep)$-restricted.
	\end{conjecture}
	Unfortunately, the following proposition\footnote{We remark that Alex Scott (personal communication) independently discovered similar counterexamples.} refutes this conjecture in a strong sense.
	\begin{proposition}
		\label{prop:counter}
		Let $N\ge1$. 
		Then for all integers $m,n$ with $n\ge m\ge20N^2$, every $\vep\in(0,\frac{1}{18})$, and every graph $H$ with $h:=\abs H\ge2$,
		there is a graph on $n$ vertices which has at most $hmn^{h-1}$ copies of $H$ and is not $(N,\vep)$-restricted.
		In particular, for every $\kappa>0$ and every integer $n\ge 20\kappa^{-1}hN^2$,
		there is a graph on $n$ vertices which has at most $\kappa n^h$ copies of $H$ and is not $(N,\vep)$-restricted.
	\end{proposition}
	\begin{proof}
		In what follows, $\Delta(G)$ denotes the maximum degree of a graph $G$.
		By taking complements if necessary, we may assume $H$ is connected,
		and so $H$ has at least one edge as $h\ge2$.
		
		Let $F$ be a random graph on $m\ge20N^2$ vertices where each edge appears independently with probability~$\frac{1}{2}$.
		For every $T\subset V(F)$ with $\abs{T}\ge\frac{1}{N}m$,
		since $6\vep<\frac{1}{3}$,
		Hoeffding's inequality~\cite{MR144363} implies that $T$ is weakly $6\vep$-restricted in $F$ with probablity at most $2\exp(-\frac{1}{72}{\abs{T}\choose2})
		\le 2\exp(-\frac{1}{300N^2}m^2)$;
		and so, since $2^m\cdot 2\exp(-\frac{1}{300N^2}m^2)<1$ (as $m\ge20N^2$),
		there is a choice of $F$ with no weakly $6\vep$-restricted set of size at least $\frac{1}{N}m$.
		Consequently $F$ has no $3\vep$-restricted subset of size at least $\frac{1}{N}m$.
		
		Now, fix such an $F$; and for every $n\ge m$, let $G$ be a graph obtained from $F$ by adding $n-m$ isolated vertices and making each of them adjacent to every vertex in $V(F)$.
		Since $H$ has at least one edge, every copy of $H$ in $G$ has at least one image vertex in $V(F)$, and so
		\[\ind_H(G)
		\le\sum_{i=1}^h{h\choose i}m^i(n-m)^{h-i}
		=n^h-(n-m)^h
		=m\sum_{i=1}^{h}n^{i-1}(n-m)^{h-i}
		\le hmn^{h-1}.\]
		
		It thus remains to show that $G$ is not $(N,\vep)$-restricted.
		Suppose not; and let $A_1\cup\cdots\cup A_k$ be a partition of $V(G)$ for some $k\le N$ such that $A_i$ is $\vep$-restricted for all $i\in\{1,2,\ldots,k\}$.
		Then $\bigcup_{i=1}^k(A_i\cap V(F))$ is a partition of $V(F)$,
		and so we may assume $T:=A_1\cap V(F)$ has size at least $\frac{1}{N}m$.
		Thus $T$ is not $3\vep$-restricted in $F$;
		hence $S:=A_1\setminus V(F)$ is nonempty.
		It follows that
		\begin{align*}
			\Delta(G[A_1])
			&=\abs{S}+\Delta(F[T])
			>\vep\abs{S}+3\vep\abs{T}
			>\vep(\abs{S}+\abs{T})
			=\vep\abs{A_1},\\
			\Delta(\overline G[A_1])
			&=\max(\abs{S}-1,\Delta(\overline F[T]))
			\ge\max(\abs{S}-1,3\vep\abs T)
			>\vep(\abs{S}+\abs{T})
			=\vep\abs{A_1}.
		\end{align*}
		Therefore $A_1$ is not $\vep$-restricted in $G$, a contradiction.
		This proves \cref{prop:counter}.
	\end{proof}
	The graphs constructed in \cref{prop:counter} suggest that an \dd exceptional\ee{} set of vertices should necessarily be removed in order for the remaining vertices to admit a partition into a bounded number of $\vep$-restricted pieces.
	Our main theorem shows that this is also sufficient.
	\begin{theorem}
		\label{thm:copies}
		For every $\vep>0$ and every graph $H$, there exist $\kappa=\kappa_{\ref{thm:copies}}(H,\vep)>0$ and $N=N_{\ref{thm:copies}}(H,\vep)>0$
		such that for every $d\ge0$ and every graph $G$ with $\ind_H(G)\le\kappa d^{\abs H}$,
		there is a set $S\subset V(G)$ with $\abs{S}\le d$ such that $G\setminus S$ is $(N,\vep)$-restricted;
		equivalently, $G$ can be made $(N,\vep)$-restricted by removing at most $C\cdot\ind_H(G)^{1/\abs H}$ vertices where $C=\kappa^{-1/\abs H}$.
	\end{theorem}
	We would like to make three remarks.
	First, \cref{thm:csss} is a special case of \cref{thm:copies} with $d=0$;
	and taking $d=\vep\abs{G}$ in \cref{thm:copies} yields \cref{thm:niki}.
	Thus \cref{thm:copies} can be viewed as a remedy for the false \cref{conj:false};
	and the counterexamples in \cref{prop:counter} (with suitable choices of $m,n$ depending on $d$ and more isolated vertices added) show that \cref{thm:copies} is optimal up to $\kappa$ and $N$ for \emph{any} given value of $d$.
	
	Second, \cref{thm:copies} is related to the induced removal lemma~\cite{MR1251840,MR4052933}
	which also implies \cref{thm:niki}. 
	Here, we are dealing with the property of being $(N,\vep)$-restricted which is weaker than $H$-freeness (by \cref{thm:csss}) and not closed under the induced subgraph relation.
	But the trade-off is worth considering:
	removing only a handful of vertices instead of adding/deleting edges;
	and working well for {\em all} graphs, including those with subquadratic number of edges and only few copies of $H$.
	
	Third, our proof of \cref{thm:copies} generalizes the proof of \cref{thm:csss} given in \cite{chudnovsky2021-1}, 
	demonstrating that the argument there can be extended to graphs with a bounded number of copies of $H$ (at the cost of removing a small number of vertices).
	The resulting bounds on $\kappa_{\ref{thm:copies}}^{-1}(H,\vep)$ and $N_{\ref{thm:copies}}(H,\vep)$, as a result, are better than what the regularity lemma could provide (but still huge functions, namely towers of twos of height depending solely on $\abs{H}$
	with $\vep^{-1}$ on top).
	It would be interesting to prove \cref{thm:copies} with bounds on $\kappa^{-1}$ and $N$ similar to the bounds obtained in~\cite{fox2008} or even in~\cite{bnss2023}.
		%
		%
	
	In what follows, for an integer $k\ge0$, let $[k]$ denote $\{1,2,\ldots,k\}$ if $k\ge1$ and $\emptyset$ if $k=0$.
	The vertex set of $H$ will always be $\{v_1,\ldots,v_h\}$ for some $h\ge1$;
	and we drop the subscript $H$ from the notation $\ind_H$.
	\section{A slight digression}
	This section provides a short and self-contained proof of \cref{thm:niki}
	without using the regularity lemma,
	which will be used frequently in the proof of \cref{thm:copies}.
	The presentation here mostly follows~\cite{fox2008}.
	
	For $\vep>0$, a graph $G$, and disjoint subsets $A,B$ of $V(G)$,
	$B$ is {\em $\vep$-sparse} to $A$ in $G$ if every vertex in $B$ is adjacent to fewer than $\vep\abs A$ vertices of $A$ in $G$,
	and {\em $\vep$-dense} to $A$ in $G$ if it is $\vep$-sparse to $A$ in $\overline G$.
	Say that $B$ is {\em $\vep$-tight} to $A$ if it is either $\vep$-sparse or $\vep$-dense to $A$.
	The following lemma implicitly appears in~\cite[Lemma 4.1]{fox2008},
	which in turn generalizes an old result of Erd\H{o}s and Hajnal~\cite[Theorem 1.5]{MR1031262}.
	This result will also be useful later on.
	\begin{lemma}
		\label{lem:counting}
		Let $H$ be a graph, and let $\vep_1,\ldots,\vep_{h-1},
		\delta_1,\ldots,\delta_{h-1}\in(0,1)$.
		Let $G$ be a graph, and let $D_1,\ldots,D_h$ be disjoint nonempty subsets of $V(G)$ such that for all indices $i,j$ with $1\le i<j\le h$,
		there do not exist $A\subset D_i$ and $B\subset D_j$
		with $\abs A\ge\prod_{t=j}^{h-1}\vep_t\cdot\abs{D_i}$
		and $\abs B\ge\frac{\delta_{j-1}}{j-1}\prod_{t=j}^{h-1}\vep_t\cdot\abs{D_j}$
		satisfying $B$ is $\vep_j$-sparse to $A$ if $v_iv_j\in E(H)$ and $\vep_j$-dense to $A$ if $v_iv_j\nin E(H)$.
		Then there are at least $\prod_{t=1}^{h-1}(1-\delta_t)\vep_t^{t}\cdot\prod_{i=1}^h\abs{D_i}$
		copies $\varphi$ of $H$ in $G$ with $\varphi(v_i)\in D_i$ for all $i\in[h]$.
	\end{lemma}
	\begin{proof}
		Induction on $h\ge1$.
		We may assume that  $h\ge2$.
		For $i\in[h-1]$, let $P_i$ be the set of vertices in $D_h$ with fewer than $\vep_{h-1}\abs{D_i}$ neighbors in $D_i$
		if $v_iv_h\in E(H)$
		and the set of vertices in $D_h$ with fewer than $\vep_{h-1}\abs{D_i}$ nonneighbors in $D_i$ if $v_iv_h\nin E(H)$.
		By the hypothesis, $\abs{P_i}\le\frac{\delta_{h-1}}{h-1}\abs{D_h}$ for all $i\in[h-1]$.
		Let $D_h':=D_h\setminus(\bigcup_{i\in[h-1]}P_i)$;
		then $\abs{D_h'}\ge (1-\delta_{h-1})\abs{D_h}$.
		
		Now, for each $u\in D_h'$ and $i\in[h-1]$,
		let $D_i^u$ be the set neighbors of $u$ in $D_i$ if $v_iv_h\in E(H)$
		and the set of nonneighbors of $u$ in $D_i$ if $v_iv_h\nin E(H)$;
		then $\abs{D_i^u}\ge\vep_{h-1}\abs{D_i}$ for all $i\in[h-1]$.
		Thus for all indices $i,j$ with $1\le i<j\le h-1$,
		there do not exist $A\subset D_i^u$ and $B\subset D_j^u$ with $\abs{A}\ge\prod_{t=j}^{h-2}\vep_t\cdot\abs{D_i^u}$
		and $\abs B\ge\frac{\delta_{j-1}}{j-1}\prod_{t=j}^{h-2}\vep_t\cdot\abs{D_j^u}$
		such that $B$ is $\vep_j$-sparse to $A$ if $v_iv_j\in E(H)$ and $\vep_j$-dense to $A$ if $v_iv_j\nin E(H)$.
		So by induction, there are at least
		$\prod_{t=1}^{h-2}(1-\delta_t)\vep_t^{t}\cdot\prod_{i=1}^{h-1}\abs{D_i^u}$ copies $\varphi_u$ of $H\setminus v_h$ in $G\setminus D_h$ with $\varphi_u(v_i)\in D_i^u$ for all $i\in[h-1]$.
		Summing up over all $u\in D_h'$, we deduce that there are at least
		\[\sum_{u\in D_h'}
		\left(\prod_{t=1}^{h-2}(1-\delta_t)\vep_t^t
		\cdot\prod_{i=1}^{h-1}\abs{D_i^u}\right)
		\ge\abs{D_h'}
		\left(\prod_{t=1}^{h-2}(1-\delta_t)\vep_t^{t}\right)
		\left(\vep_{h-1}^{h-1}\prod_{i=1}^{h-1}\abs{D_i}\right)
		\ge\prod_{t=1}^{h-1}(1-\delta_t)\vep_t^t
		\cdot\prod_{i=1}^h\abs{D_i}\]
		copies $\varphi$ of $H$ in $G$ such that $\varphi(v_i)\in D_i$ for all $i\in[h]$.
		This proves~\cref{lem:counting}.
	\end{proof}
	\begin{corollary}
		\label{cor:counting}
		Let $\vep\in(0,1)$, let $H$ be a graph,
		and let $\kappa=\kappa_{\ref{cor:counting}}(H,\vep):=(4h)^{-h}\vep^{{h\choose2}}$.
		Then every $G$ with $\ind(G)\le \kappa\abs{G}^h$
		contains disjoint $A,B\subset V(G)$ with $\abs{A},\abs{B}\ge(2h)^{-2}\vep^{h-1}\abs{G}$ such that $B$ is $\vep$-tight to $A$.
	\end{corollary}
	\begin{proof}
		We may assume $\abs{G}\ge h$.
		Let $D_1,\ldots,D_h$ be disjoint subsets of $V(G)$ each of size $\floor{\frac{1}{h}\abs{G}}$;
		then $\abs{D_t}\ge\frac{1}{2h}\abs{G}$ for all $t\in[h]$.
		It suffices to apply \cref{lem:counting}
		with $\vep_t=\vep$ and $\delta_t=\frac{1}{2}$ for all $t\in[h]$.
	\end{proof}
	For $\vep_1,\vep_2,\kappa>0$ and a graph $H$,
	let $\beta(H,\kappa,\vep_1,\vep_2)$ be the largest constant $\beta$ with $0<\beta\le1$ such that every graph $G$ with $\ind(G)\le\kappa\abs{G}^h$ has an induced subgraph with at least $\beta\abs{G}$ vertices and edge density at most $\vep_1$ or at least $1-\vep_2$;
	then $\beta(H,\kappa,\vep_1,\vep_2)$ is decreasing in~$\kappa$ and
	$\beta(H,\kappa,\vep_1,\vep_2)=1$ for all $\kappa>0$ whenever $\vep_1+\vep_2\ge1$
	(and so whenever $\vep_1\vep_2\ge1$).
	We need the following lemma.
	\begin{lemma}
		\label{lem:niki}
		Let $\vep_1,\vep_2>0$,
		let $H$ be a graph, and let $\eta:=\eta_{\ref{lem:niki}}(H,\vep_1,\vep_2):=\frac{1}{2}(2h)^{-2}(\frac{1}{4}\vep)^{h-1}$ where $\vep=\min(\vep_1,\vep_2)$.
		Then for every $\kappa$ with $0<\kappa\le \kappa_{\ref{cor:counting}}(H,\frac{1}{4}\vep)$, we have
		\[\beta(H,\kappa,\vep_1,\vep_2)
		\ge\eta\cdot \min(\beta(H,\eta^{-h}\kappa,\textstyle{\frac{3}{2}\vep_1},\vep_2),
		\beta(H,\eta^{-h}\kappa,\vep_1,\textstyle{\frac{3}{2}}\vep_2)).\]
	\end{lemma}
	\begin{proof}
		Let
		$\beta_1:=\beta(H,\eta^{-h}\kappa,\frac{3}{2}\vep_1,\vep_2)$, 
		$\beta_2:=\beta(H,\eta^{-h}\kappa,\vep_1,\frac{3}{2}\vep_2)$,
		and $\beta_0:=\eta\cdot\min(\beta_1,\beta_2)$.
		Let $G$ be a graph with $\ind(G)\le\kappa\abs{G}^h$;
		we need to show there is a set $S\subset V(G)$ with $\abs S\ge\beta_0\abs G$ such that $G[S]$ has edge density at most $\vep_1$ or at least $1-\vep_2$.
		By \cref{cor:counting},
		$G$ has disjoint subsets $A,B\subset V(G)$ with $\abs{A},\abs{B}\ge2\eta\abs{G}$
		such that $B$ is $\frac{1}{4}\vep$-tight to $A$;
		and we may assume $B$ is $\frac{1}{4}\vep$-sparse to $A$.
		
		Because $\ind(G[B])\le\kappa\abs{G}^h\le \eta^{-h}\kappa\abs{B}^h$,
		by the definition of $\beta$ and by averaging, 
		there exists $B_1\subset B$ with $\abs{B_1}=\ceil{\beta_1\eta\abs{G}}\ge\beta_0\abs{G}$
		such that $G[B_1]$ has edge density at most $\frac{3}{2}\vep_1$ or at least $1-\vep_2$.
		If the latter holds then we are done,
		so we may assume the former holds.
		
		Let $A_0$ be the set of vertices in $A$ each with at most $\frac{1}{2}\vep\abs{B_1}$ neighbors in $B_1$.
		Since $G$ has fewer than $\frac{1}{4}\vep\abs{A}\abs{B_1}$ edges between $A$ and $B_1$,
		we have $\abs{A_0}\ge\frac{1}{2}\abs{A}\ge\eta\abs{G}$.
		Thus $\ind(G[A_0])\le \eta^{-h}\kappa\abs{A_0}^h$,
		and so by the definition of $\beta$ and by averaging,
		there exists $A_1\subset A_0$ with $\abs{A_1}=\ceil{\beta_1\eta\abs{G}}\ge\beta_0\abs{G}$ such that
		$G[A_1]$ has edge density at most $\frac{3}{2}\vep_1$ or at least $1-\vep_2$.
		Again, we may assume the former holds.
		
		Now, let $S:=A_1\cup B_1$;
		then $\abs{S}=2\abs{A_1}=2\abs{B_1}\ge2\beta_0\abs{G}$.
		Since $G[A_1],G[B_1]$ each have edge density at most $\frac{3}{2}\vep_1$ and $G$ has at most $\frac{1}{2}\vep\abs{A_1}\abs{B_1}$ edges between $A_1$ and $B_1$, we deduce that
		\begin{align*}
			\abs{E(G[S])}
			&\le\abs{E(G[A_1])}+\abs{E(G[B_1])}+\frac{1}{2}\vep\abs{A_1}\abs{B_1}\\
			&\le\frac{3}{2}\vep_1{\abs{A_1}\choose2}
			+\frac{3}{2}\vep_1{\abs{B_1}\choose2}
			+\frac{1}{2}\vep_1\abs{A_1}\abs{B_1}
			=3\vep_1{\abs{A_1}\choose2}+\frac{1}{2}\vep_1\abs{A_1}^2
			\le\vep_1{2\abs{A_1}\choose2}
			=\vep_1{\abs{S}\choose2}.
		\end{align*}
		Therefore $S$ has the desired property.
		This proves \cref{lem:niki}.
	\end{proof}
	We are now give a proof of \cref{thm:niki}
	in the following equivalent form,
	which leads to the dependence of $\delta_{\ref{thm:niki}}(H,\vep)$ and $\kappa_{\ref{thm:niki}}(H,\vep)$ on $\vep$ and $h$ as mentioned in the introduction.
	\begin{theorem}
		\label{thm:nikiweak}
		For every $\vep>0$ and every graph $H$, there exist $\delta=\delta(H,\vep)>0$ and $\kappa=\kappa(H,\vep)>0$ such that every graph $G$
		with $\ind(G)\le\kappa\abs{G}^h$ contains a weakly $\vep$-restricted set of size at least $\delta\abs G$.
	\end{theorem}
	\begin{proof}
		Let $s:=\ceil{\log_{\frac{3}{2}}(\vep^{-2})}$, $\eta:=\eta_{\ref{lem:niki}}(H,\vep,\vep)$,
		$\delta:=\eta^s$,
		and $\kappa:=\eta^{sh}\cdot \kappa_{\ref{cor:counting}}(H,\frac{1}{4}\vep)$.
		Note that $\eta_{\ref{lem:niki}}(H,\cdot,\cdot)$ is decreasing in each of the last two components.
		Thus, since $\beta(H,\cdot,\cdot,\cdot)$ is decreasing in the second component
		and equals $1$ whenever the last two components have product at least $1$,
		applying \cref{lem:niki} for $s$ times yields $\beta(H,\kappa,\vep,\vep)\ge\eta^s=\delta$.
		This proves \cref{thm:nikiweak}.
	\end{proof}
	\section{Key lemma}
	This section introduces and proves our key lemma, the following.
	\begin{lemma}
		\label{lem:main}
		For all $\vep,\eta,\theta\in(0,\frac{1}{2})$ and every graph $H$,
		there are $\kappa=\kappa_{\ref{lem:main}}(H,\vep,\eta,\theta)>0$
		and $N=N_{\ref{lem:main}}(H,\vep,\eta,\theta)>0$ with the following property.
		For every $d\ge0$ and every graph $G$ with $\ind(G)\le\kappa d^h$,
		there is a set $S\subset V(G)$ with $\abs{S}\le d$ such that $V(G)\setminus S$ can be partitioned into nonempty sets
		\[A_1,\ldots,A_m;\,B_1,\ldots,B_m;\,C_1,\ldots,C_n\]
		where $m\le {h\choose2}$ and $n\le N$, such that
		\begin{itemize}
			\item $A_1,\ldots,A_m,C_1,\ldots,C_n$ are $\vep$-restricted in $G$; and
			
			\item for every $i\in[m]$, $\abs{B_i}\le\eta\abs{A_i}$ and $B_i$ is $\theta$-tight to $A_i$.
		\end{itemize}
	\end{lemma}
	This contains~\cite[Theorem 1.5]{chudnovsky2021-1} as a special case with $d=0$, and already gives \cref{thm:niki} with $\vep=\eta=\theta$ and $d=\vep\abs G$.
	We shall employ the same approach as in~\cite[Section 2]{chudnovsky2021-1},
	and recommend reading the detailed sketch there first.
	Here we explain the modifications.
	
	We recall some definitions.
	For $c,\vep>0$ and a graph $G$, a pair $(A,B)$ of disjoint nonempty subsets of $V(G)$ is {\em $(c,\vep)$-full} in $G$ if for every $A_1\subset A$ and $B_1\subset B$ with $\abs{ A_1}\ge c\abs A$ and $\abs{B_1}\ge c\abs B$,
	$G$ has at least $\vep\abs {A_1}\abs {B_1}$ edges between $A_1,B_1$;
	and $(A,B)$ is {\em $(c,\vep)$-empty} in $G$ if it is $(c,\vep)$-full in $\overline{G}$.
	Thus for every $c'>c$ and every $A'\subset A$ and $B'\subset B$ with $\abs {A'}\ge c'\abs{A}$ and $\abs{B'}\ge c'\abs{B}$,
	$(A',B')$ is $(c/c',\vep)$-full if $(A,B)$ is $(c,\vep)$-full and is $(c/c',\vep)$-empty if $(A,B)$ is $(c,\vep)$-empty.
	A collection $\{D_1,\ldots,D_h\}$ of disjoint nonempty subsets of $V(G)$ is a {\em $(c,\vep)$-blowup} of $H$ if for all distinct $i,j\in[h]$,
	$(D_i,D_j)$ is $(c,\vep)$-full if $v_iv_j\in E(H)$
	and is $(c,\vep)$-empty if $v_iv_j\nin E(H)$.
	
	In proving~\cref{lem:main}, we shall be concerned with partitions of $V(G)$ into \dd rows\ee{} of subsets and pairs of subsets~as~follows:
	\begin{itemize}
		\item {\em first row}: pairs $(A_1,B_1),\ldots,(A_m,B_m)$ for some $m\ge0$
		such that for all $i\in[m]$, $A_i$ is $\vep$-restricted,
		$B_i$ is very tight to $A_i$ and has size smaller than a tiny fraction of $A_i$
		($B_i$ might be empty);
		
		\item {\em second row}: $\vep$-restricted nomempty sets $C_1,\ldots,C_n$ for some $n\ge0$;
		
		\item {\em third row}: 
		$\vep'$-restricted nonempty sets $D_1,\ldots,D_t$
		for some $t$ with $0\le t\le h$, such that $\{D_1,\ldots,D_t\}$ is a $(c,\xi)$-blowup of $H[\{v_1,\ldots,v_t\}]$
		for some appropriately chosen $c,\vep',\xi>0$; and
		
		\item {\em fourth row}: the set $L$ of \dd leftover\ee{} vertices such that whenever $t>0$,
		$L$ has size smaller than a tiny fraction of each $D_i$.
	\end{itemize}
	
	Such a partition certainly exists, with $m=n=t=0$ and $L=V(G)$.
	Starting from $t=0$ with this partition, we shall attempt to increase $t$ one by one for at most $h$ steps.
	Let $S$ be the set of vertices in $L$ with the \dd correct adjacencies\ee{} to the collection $\{D_1,\ldots,D_t\}$,
	that is, those having at least a small fraction of neighbors in $D_i$ if $v_{t+1}v_i\in E(H)$ and at least a small fraction of nonneighbors in $D_i$ if $v_{t+1}v_i\nin E(H)$.
	Then $L\setminus S$ can be partitioned into (possibly empty) sets $L_1,\ldots,L_t$ such that $L_i$ is (very) tight to $D_i$ for every $i\in[t]$.
	As the notation suggests, if $\abs S\le d$ then we stop the iteration and rearrange the sets $A_1,\ldots,A_m$, $B_1,\ldots,B_m$, $C_1,\ldots,C_n$, $D_1,\ldots,D_t$, $L_1,\ldots,L_t$
	to form a partition of $V(G)\setminus S$ with the desired property (this is not hard, and the bounds on $m$ and $n$ will come up later).
	
	So let us assume $\abs S>d$.
	We can then apply \cref{thm:niki} to find an $\vep'$-restricted subset $S_0$ of~$S$.
	Keeping in mind that $\{D_1,\ldots,D_t,S_0\}$ now form a \dd partial\ee{} blowup of $H[\{v_1,\ldots,v_t,v_{t+1}\}]$,
	we iteratively construct a nested sequence $S_0\supset S_1\supset\ldots\supset S_t$ and subsets $P_1\subset D_1,\ldots,P_t\subset D_t$
	such that each pair $(S_i,P_i)$ is reasonably full (if $v_{t+1}v_i\in E(H)$) or reasonably empty (if $v_{t+1}v_i\nin E(H)$);
	then the collection $\{P_1,\ldots,P_t,S_{t}\}$ will be a sufficiently good blowup of $H[\{v_1,\ldots,v_t,v_{t+1}\}]$
	while $P_1,\ldots,P_t,S_t$ are still $\vep'$-restricted (for suitable $c,\vep',\xi$).
	To execute this process, we need the following useful theorem of Yuejian, R\"odl, and Ruci\'nski~\cite[Theorem 1.3]{MR1887082} which allows us to extract decent fullness/emptiness from moderate denseness/sparseness.
	(We remark that~\cite[Theorem 1.3]{MR1887082} was stated only for balanced bipartite graphs; but the proof there works equally well for unbalanced ones.)
	\begin{lemma}
		[\cite{MR1887082}]
		\label{lem:full}
		Let $c\in(0,1)$ and $\vep\in(0,\frac14)$. Then, for  $\gamma=\gamma_{\ref{lem:full}}(c,\vep):=\frac12(2\vep)^{12/c}\in(0,\frac{1}{3})$,
		the following holds. 
		Let $G$ be a graph with $A,B\subset V(G)$ disjoint and nonempty such that $G$ has at least $2\vep\abs{A}\abs B$ edges between $A$ and $B$.
		Then there exist $A'\subset A$ and $B\subset B'$ with $\abs{A'}\ge\gamma\abs{A}$ and $\abs{B'}\ge\gamma\abs{B}$ such that $(A',B')$ is $(c,\vep)$-full.
	\end{lemma}
	Observe that $L$ is nonempty since $S$ is, which implies each $D_i$ is quite large,
	and so we can take each $P_i$ to have size at least a (small) fraction of $D_i$ yet at most half of $D_i$ simultaneously.
	Then each $L_i$ is still quite tight to and tiny compared to $D_i\setminus P_i$;
	and we can move each pair $(D_i\setminus P_i,L_i)$ to the first~row.
	
	Now, we want to use \cref{thm:niki} to pull out as many $\vep$-restricted sets as possible from $S\setminus S_t$ (assuming this is nonempty)
	so that the resulting new \dd leftover\ee{} set $L'$ still has size smaller than a tiny fraction of $S_t$ and of each $P_i$;
	then we can move those new restricted sets to the third row.
	A potential issue here is that \cref{thm:niki} may not be applicable if $S\setminus S_t$ is not large enough while most of the copies of $H$ in $G$ are \dd concentrated\ee{} on $G[S\setminus S_t]$.
	This can be avoided, conveniently, by making sure that $\abs{S_0}$ is not too large compared to $\abs S$ right in the first place (if $\abs S\ge2$), which will be done by the following simple corollary of \cref{thm:niki} itself
	(we believe this is well-known, but still include a proof for completeness).
	\begin{corollary}
		\label{cor:niki1}
		For every $\vep>0$ and every graph $H$, there exist $\delta=\delta_{\ref{cor:niki1}}(H,\vep)\in(0,\frac{1}{4})$ and $\kappa=\kappa_{\ref{cor:niki1}}(H,\vep)>0$ such that for every graph $G$ with $\ind(G)\le\kappa\abs{G}^h$,
		$G$ has an $\vep$-restricted set $T$ with $\abs T=\ceil{\delta\abs G}$;
		in particular $\abs{T}=1$ if $\abs{G}=1$ and $\abs{G\setminus T}\ge\frac{1}{2}\abs{G}$ if $\abs{G}\ge2$.
	\end{corollary}
	\begin{proof}
		Let $\delta:=\frac{1}{4}\cdot\delta_{\ref{thm:niki}}(H,\frac{1}{8}\vep)$
		and $\kappa:=\kappa_{\ref{thm:niki}}(H,\frac{1}{8}\vep)$.
		By~\cref{thm:niki}, $G$ has an $\frac{1}{8}\vep$-restricted set $U$ with $\abs U\ge2\delta\abs{G}$;
		in particular $U$ is weakly $\frac{1}{4}\vep$-restricted.
		By averaging, there is a weakly $\frac{1}{4}\vep$-restricted subset $U'$ of $U$ such that $\abs {U'}=\ceil{2\delta\abs G}$,
		and so there is $T\subset U'$ with $\abs{T}=\ceil{\frac{1}{2}\abs{U'}}
		=\ceil{\frac{1}{2}\ceil{2\delta\abs G}}
		=\ceil{\delta\abs{G}}$
		such that $T$ is $\vep$-restricted in $G$.
		In particular, if $\abs{G}\le4$ then $\abs{T}=1$;
		and if $\abs{G}>4$ then
		$\abs{G\setminus T}> \abs{G}-1-\delta\abs{G}
		\ge\frac{3}{4}\abs{G}-\frac{1}{4}\abs{G}=\frac{1}{2}\abs{G}$.
		This proves \cref{cor:niki1}.
	\end{proof}
	For $\delta,\eta\in(0,1)$, let $\phi(\delta,\eta)$ be the least integer $p\ge1$ with $(1-\delta)^p\le\eta$;
	then $\phi(\delta,\eta)\le\delta^{-1}\log\eta$.
	The next corollary of \cref{thm:niki} formalizes the process of repeatedly pulling out $\vep$-restricted sets from~$S\setminus S_t$.
	\begin{corollary}
		\label{cor:niki2}
		For every $\vep,\eta\in(0,1)$, for every graph $H$,
		and for $\delta:=\delta_{\ref{thm:niki}}(H,\vep)>0$,
		there exists $\kappa=\kappa_{\ref{cor:niki2}}(H,\vep,\eta)>0$
		such that for every graph $G$ with $\ind(G)\le\kappa\abs{G}^h$,
		there is $T\subset V(G)$
		with $\abs{T}\le\eta\abs{G}$
		such that $G\setminus T$ is $(\phi(\delta,\eta),\vep)$-restricted.
	\end{corollary}
	\begin{proof}
		Let $\kappa:=\eta^{h}\cdot\kappa_{\ref{thm:niki}}(H,\vep)$.
		We may assume $\abs G\ge1$.
		Let $U_0:=V(G)$;
		and for $i\ge0$, as long as $U_i$ is defined and $\abs{U_i}>\eta\abs G$, let $U_{i+1}\subset U_i$ such that $U_{i}\setminus U_{i+1}$ is $\vep$-restricted and $\abs{U_i\setminus U_{i+1}}\ge\delta\abs{U_i}$,
		which is possible by \cref{thm:niki}~since \[\ind(G[U_i])
		\le \kappa\abs{G}^h= \kappa_{\ref{thm:niki}}(H,\vep)\cdot(\eta\abs{G})^h
		< \kappa_{\ref{thm:niki}}(H,\vep)\cdot\abs{U_i}^h.\]
		This produces a chain of sets $V(G)=U_0\supset U_1\supset\ldots\supset U_n$ for some $p\ge1$
		such that $\abs{U_{i+1}}\le (1-\delta)\abs{U_{i}}\le(1-\delta)^{i+1}\abs{G}$ 
		and $\abs{U_i}>\eta\abs{G}$ for all $i\in\{0,1,\ldots,p-1\}$.
		In particular $\eta\abs{G}<\abs{U_{n-1}}\le(1-\delta)^{n-1}\abs{G}$;
		thus $p-1<\phi(\delta,\eta)$ and so $p\le \phi(\delta,\eta)$.
		Let $T:=U_p$;
		then $\bigcup_{i=1}^{p}(U_i\setminus U_{i-1})$ is a partition of $V(G)\setminus T$ into $p$ subsets which are $\vep$-restricted in $G$.
		This proves \cref{cor:niki2}.
	\end{proof}
	Now assume we have reached $t=h$ and obtained a decent blowup $\{D_1,\ldots,D_h\}$ of $H$.
	Observe that to be able to reach $t=h$ means the \dd exceptional\ee{} set $S$ in each step always had size more than $d$;
	so it is not hard to see that each $\abs{D_i}$ is still more than a (tiny) fraction of $d$.
	It thus suffices to apply the following,
	which is a direct corollary of \cref{lem:counting} and is an analogue of the induced counting lemma~\cite[Lemma~3.2]{MR1804820}.
	\begin{corollary}
		\label{cor:copies}
		Let $\vep\in(0,\frac{1}{2})$, let $H$ be a graph, and
		let $G$ be a graph with an $(\vep^{h},\vep)$-blowup $\{D_1,\ldots,D_h\}$ of $H$.
		Then there are at least $(1-\vep)^{h-1}\vep^{{h\choose2}}\abs {D_1}\cdots\abs {D_h}$ copies $\varphi$ of $H$ in $G$ with $\varphi(v_i)\in D_i$ for all $i\in[h]$.
	\end{corollary}
	\begin{proof}
		This follows from \cref{lem:counting} with
		$\vep_t:=\vep$ and $\delta_t:=t\cdot\vep^t$ for all $t\in[h-1]$;
		note that $\delta_t\le t2^{-t+1}\vep\le \vep$ since $\vep\in(0,\frac{1}{2})$.
	\end{proof}
	We are now ready to prove \cref{lem:main}.
	\begin{proof}
		[Proof of \cref{lem:main}]
		Let $\xi:=\frac{1}{4}\theta$ and $\vep_h:=\min(\vep,\xi^{h})$.
		Let $\Gamma_{t,t}=\lambda_{t,t}:=1$;
		and for $t=h-1,h-2,\ldots,0$ in turn,
		do the following:
		\begin{itemize}
			\item for $i=t-1,t-2,\ldots,0$ in turn,
			let $\Gamma_{t,i}:=\lambda_{t,i+1}\Gamma_{t,i+1}$ and $\lambda_{t,i}:=\gamma_{\ref{lem:full}}(\frac{1}{3}\vep_{t+1}\Gamma_{t,i+1},\xi)$; and
			
			\item let $\vep_t:=\vep_{t+1}\lambda_{t,0}$.
		\end{itemize}
		Now, define
		\begin{gather*}
			\vep':=\min_{t\in\{0,1,\ldots,h-1\}}\vep_{t+1}\Gamma_{t,0},
			\qquad 
			\delta':=\delta_{\ref{cor:niki1}}(H,\vep'),\\
			\eta':=\frac{1}{2}\eta\delta'\cdot\min_{t\in\{0,1,\ldots,h-1\}}\Gamma_{t,0},
			\qquad
			N:={h\choose2}+(h-1)\cdot\phi(\delta',\eta').
		\end{gather*}
		Also, for $i=1,2,\ldots,h$ in turn, do the following:
		\begin{itemize}
			\item let $\Lambda_{i,i}:=\delta'\Gamma_{i-1,0}$; and
			
			\item for $t=i,i+1,\ldots,h-1$ in turn, let
			$\Lambda_{t+1,i}:=\lambda_{t,i}\Lambda_{t,i}$.
		\end{itemize}
		Finally, put
		\[\kappa:=\min\left((1-\xi)^{h-1}\xi^{{h\choose2}}\Lambda_{h,1}\cdots\Lambda_{h,h},\,
		\kappa_{\ref{cor:niki1}}(H,\vep'),\,
		2^{-h}\cdot\kappa_{\ref{cor:niki2}}(H,\vep,\eta')\right).\]
		For integers $m,n,t\ge0$ with $t\le h$,
		an {\em $(m,n,t)$-partition} in $G$ is a partition of $V(G)$ into (not necessarily nonempty) subsets
		\[A_1,\ldots,A_m;\,B_1,\ldots,B_m;\,C_1,\ldots,C_n;\,D_1,\ldots,D_t;\,L\]
		such that
		\begin{itemize}
			\item $m\le{t\choose2}$ and $n\le t\cdot\phi(\delta',\eta')$;
			
			\item $A_1,\ldots,A_m,C_1,\ldots,C_n$ are nonempty and $\vep$-restricted;
			
			\item for every $i\in[m]$, $\abs{B_i}\le\eta\abs{A_i}$ and $B_i$ is $\theta$-tight to $A_i$;
			
			\item $\{D_1,\ldots,D_t\}$ is an $(\vep_t,\xi)$-blowup of $H[\{v_1,\ldots,v_t\}]$; and
			
			\item if $t>0$, then $\abs{D_i}>\max(\Lambda_{t,i}d,2\eta^{-1}\abs{L})$ and $D_i$ is $\vep_t$-restricted for every $i\in[t]$.
		\end{itemize}
		For the readers' convenience, let us write such a partition as follows
		\begin{gather*}
			(A_1,B_1),\ldots,(A_m,B_m);\\
			C_1,\ldots,C_n;\\
			D_1,\ldots,D_t;\\
			L.
		\end{gather*}
		
		Observe that $V(G)$ itself is a $(0,0,0)$-partition in $G$.
		Thus, there is $t\in\{0,1,\ldots,h\}$ maximal such that there is an $(m,n,t)$-partition in $G$.
		If $t=h$, then $\{D_1,\ldots,D_h\}$ would be a $(\xi^{h-1},\xi)$-blowup of $H$;
		so by \cref{cor:copies}, $G$ would contain at least
		\[(1-\xi)^h\xi^{{h\choose2}}\abs{D_1}\cdots\abs{D_h}
		> (1-\xi)^h\xi^{{h\choose2}}\Lambda_{h,1}\cdots\Lambda_{h,h}d^h
		\ge \kappa d^h\ge\ind(G)\]
		copies of $H$, a contradiction.
		Thus $t<h$.
		
		Let $S$ be the set of vertices $u$ in $L$ with the property that for every $i\in[t]$, $u$ has at least $2\xi\abs{D_i}$ neighbors in $D_i$ if $v_{t+1}v_i\in E(H)$ and at least $2\xi\abs{D_i}$ nonneighbors in $D_i$ if $v_{t+1}v_i\nin E(H)$.
		Then there is~a~partition $L\setminus S=L_1\cup\cdots\cup L_t$
		such that $L_i$ is $2\xi$-tight to $D_i$ for all $i\in[t]$.
		We shall prove that $S$ satisfies the lemma;
		and the following crucial claim is the key step.
		\begin{claim}
			\label{claim:main}
			$\abs{S}\le d$.
		\end{claim}
		\begin{subproof}
			Suppose that $\abs{S}>d$; then $\abs{S}\ge1$.
			Since $\delta'=\delta_{\ref{cor:niki1}}(H,\vep')$
			and
			\[
			\ind(G[S])\le \kappa d^h<\kappa_{\ref{cor:niki1}}(H,\vep')\cdot \abs{S}^h\]
			by the definitions of $\delta'$ and $\kappa$,
			\cref{cor:niki1} yields an $\vep'$-restricted subset $S_0$ of $S$ with $\abs{S_0}=\ceil{\delta'\abs{G}}$;
			in particular $\abs{S_0}=1$ if $\abs{S}=1$
			and $\abs{S\setminus S_0}\ge\frac{1}{2}\abs{S}>\frac{1}{2}d$ if $\abs{S}\ge2$.
			We shall define a chain of sets $S_0\supset S_1\supset\ldots\supset S_t$
			together with sets $P_1,\ldots,P_t$ of vertices such that for all $i\in[t]$,
			\begin{itemize}
				\item $\abs{S_i}\ge\lambda_{t,i}\abs{S_{i-1}}$;
				
				\item $P_i\subset D_i$ and 
				$\lambda_{t,i}\abs{D_i}\le\abs{P_i}\le\frac{1}{2}\abs{D_i}$; and
				
				\item $(S_i,P_i)$ is $(\vep_{t+1}\Gamma_{t,i},\xi)$-full if $v_iv_{t+1}\in E(H)$
				and is $(\vep_{t+1}\Gamma_{t,i},\xi)$-empty if $v_iv_{t+1}\nin E(H)$. 
			\end{itemize}
			
			To this end, assume that for $i\in[t]$, $S_{i-1}$ and $P_{i-1}$ have been defined.
			Assume $v_iv_{t+1}\in E(H)$ without loss of generality;
			then by the choice of $L$,
			every vertex of $S_{i-1}\subset S$ has at least $2\xi\abs{D_i}$ neighbors in $D_i$.
			Thus \cref{lem:full} and the definition of $\lambda_{t,i}$ yield $S_i\subset S_{i-1}$
			and $D_i'\subset D_i$ with $\abs{S_i}\ge\lambda_{t,i}\abs{S_{i-1}}$ and $\abs{D_i'}\ge\lambda_{t,i}\abs{D_i}$ 
			so that $(S_i,D_i')$ is $(\frac{1}{3}\vep_{t+1}\Gamma_{t,i},\xi)$-full.
			Let $P_i\subset D_i'$ with $\abs{P_i}=\min(\abs{D_i'},\floor{\frac{1}{2}\abs{D_i}})$;
			then since $\floor{\frac{1}{2}\abs{D_i}}\ge\frac{1}{3}\abs{D_i}\ge\lambda_{t,i}\abs{D_i}$
			(as $\abs{D_i}>2\eta^{-1}\abs{S}>1$),
			$\max(\lambda_{t,i}\abs{D_i},\frac{1}{3}\abs{D_i'})\le\abs{P_i}\le\frac{1}{2}\abs{D_i}$.
			In particular, $(S_i,P_i)$ is $(\vep_{t+1}\Gamma_{t,i},\xi)$-full.
			This defines $S_i$ and $P_i$ in the case $v_iv_{t+1}\in E(H)$;
			and similar arguments with $(\frac{1}{3}\vep_{t+1}\Gamma_{t,i},\xi)$-full replaced by $(\frac{1}{3}\vep_{t+1}\Gamma_{t,i},\xi)$-empty also define $S_i$ and $P_i$ in the case $v_iv_{t+1}\nin E(H)$.
			
			From the above construction, we see that $\abs{S_t}=1$ if $\abs{S}=1$ and $\abs{S\setminus S_t}\ge\abs{S\setminus S_0}>\frac{1}{2}d$ if $\abs{S}\ge2$.
			Let $L':=\emptyset$ if the former holds;
			and if the latter holds, then since
			\[\ind(G[S\setminus S_t])
			\le \kappa d^h
			<\kappa_{\ref{cor:niki2}}(H,\vep,\eta')\cdot \abs{S\setminus S_t}^h\]
			by the definition of $\kappa$,
			\cref{cor:niki2} yields $L'\subset S\setminus S_t$ with $\abs{L'}\le\eta'\abs{S\setminus S_t}$
			such that $S\setminus(S_t\cup L')$ is $(\phi(\delta',\eta'),\vep)$-restricted.
			Thus there is always a subset $L'\subset S\setminus S_t$ with $\abs{L'}\le\eta'\abs{S\setminus S_t}$
			such that $S\setminus(S_t\cup L')$ has a partition into nonempty $\vep$-restricted subsets $Q_1,\ldots,Q_s$ for some $s\le\phi(\delta',\eta')$.
			
			Now, let $P_{t+1}:=S_t$; we shall prove that the following partition of $V(G)$
			\begin{gather*}
				(A_1,B_1),\ldots,(A_m,B_m),
				(D_1\setminus P_1,L_1),\ldots,(D_t\setminus P_t,L_t);\\
				C_1,\ldots,C_n,Q_1,\ldots,Q_s;\\
				P_1,\ldots,P_t,P_{t+1};\\
				L'
			\end{gather*}
			is an $(m+t,n+s,t+1)$-partition in $G$,
			which contradicts the maximality of $t$.
			To this end, observe the following.
			\begin{itemize}
				\item Since $m\le{t\choose2}$ and $n\le t\cdot\phi(\delta',\eta')$,
				we have $m+t\le{t+1\choose2}$
				and $n+s\le (t+1)\cdot\phi(\delta',\eta')$.
				
				\item $A_1,\ldots,A_m,C_1,\ldots,C_n,Q_1,\ldots,Q_s$ are nonempty and $\vep$-restricted by definition;
				and for every $i\in[t]$, $\abs{D_i\setminus P_i}\ge\frac{1}{2}\abs{D_i}>0$ from the construction of $P_i$,
				in particular $D_i\setminus P_i$ is $\vep$-restricted
				since $D_i$ is $\vep_t$-restricted and $2\vep_t\le\vep$.
				
				\item For every $i\in[m]$, $\abs{B_i}\le\eta\abs{A_i}$ and $B_i$ is $\theta$-tight to $A_i$ by definition;
				and for every $i\in[t]$,
				$\abs{L_i}\le\abs{L}<\frac{1}{2}\eta\abs{D_i}\le\eta\abs{D_i\setminus P_i}$
				and $L_i$ is $\theta$-tight to $D_i\setminus P_i$
				since it is $2\xi$-tight to $D_i$ while $\xi=\frac{1}{4}\theta$.
				
				\item $\{P_1,\ldots,P_t\}$ is an $(\vep_{t+1},\xi)$-blowup of $H[\{v_1,\ldots,v_t\}]$ since $\abs{P_i}\ge\lambda_{t,i}\abs{D_i}
				\ge\vep_t\vep_{t+1}^{-1}\abs{D_i}$ for all $i\in[t]$
				and $\{D_1,\ldots,D_t\}$ is an $(\vep_t,\xi)$-blowup of $H[\{v_1,\ldots,v_t\}]$.
				Also,
				from the above construction, for all $i\in\{0,1,\ldots,t\}$, we have
				$\abs{P_{t+1}}=\abs{S_t}\ge(\lambda_{t,t}\lambda_{t,t-1}\cdots\lambda_{t,{i+1}})\abs{S_{i}}
				=\Gamma_{t,i}\abs{S_i}$,
				in particular $(P_i,P_{t+1})$ is $(\vep_{t+1},\xi)$-full if $v_iv_{t+1}\in E(H)$ and is $(\vep_{t+1},\xi)$-empty if $v_iv_{t+1}\nin E(H)$.
				Thus $\{P_1,\ldots,P_t,P_{t+1}\}$ is an $(\vep_{t+1},\xi)$-blowup of $H[\{v_1,\ldots,v_t,v_{t+1}\}]$.
				
				\item 
				For every $i\in[t]$,
				since $\Lambda_{t+1,i}=\lambda_{t,i}\Lambda_{t,i}$,
				$\abs{L'}\le\eta'\abs{S\setminus S_t}<\eta'\abs{S}\le\eta'\abs{L}$
				as $\abs{S_t}>0$,
				and $\eta'\le\Gamma_{t,0}<\lambda_{t,i}$ by the definition of $\eta'$, we have
				\[\abs{P_i}\ge\lambda_{t,i}\abs{D_i}
				>\lambda_{t,i}\max(\Lambda_{t,i}d,2\eta^{-1}\abs{L})
				\ge\max(\Lambda_{t+1,i}d,2\eta^{-1}\lambda_{t,i}(\eta')^{-1}\abs{L'})
				\ge\max(\Lambda_{t+1,i}d,2\eta^{-1}\abs{L'});\]
				and since $\Lambda_{t+1,t+1}=\delta'\Gamma_{t,0}$
				and $\eta'\le \frac{1}{2}\eta\delta'\Gamma_{t,0}$
				by definition,
				we deduce that
				\[\abs{P_{t+1}}
				=\abs{S_t}
				\ge \Gamma_{t,0}\abs{S_0}
				\ge \Gamma_{t,0}\delta'\abs{S}
				>\max(\Lambda_{t+1,t+1}d,\Gamma_{t,0}\delta'(\eta')^{-1}\abs{L'})
				\ge\max(\Lambda_{t+1,t+1}d,2\eta^{-1}\abs{L'}).\]
				Also, for every $i\in[t]$, $P_i$ is $\vep_{t+1}$-restricted since $D_i$ is $\vep_t$-restricted;
				and $P_{t+1}=S_t$ is $\vep_{t+1}$-restricted since $S_0$ is $\vep'$-restricted
				and $\vep'\le \vep_{t+1}\Gamma_{t,0}$ by the definition of $\vep'$.
			\end{itemize}
			
			This proves~\cref{claim:main}.
		\end{subproof}
		Now, recall that $V(G)\setminus S$ is partitioned into (possibly empty) subsets
		\[A_1,\ldots,A_m;\,B_1,\ldots,B_m;\,C_1,\ldots,C_n;\,D_1,\ldots,D_t;\,L_1,\ldots,L_t\]
		such that
		\begin{itemize}
			\item $m\le{t\choose2}$ and $n\le t\cdot \phi(\delta',\eta')$;
			
			\item $A_1,\ldots,A_m,C_1,\ldots,C_n$ are nonempty and $\vep$-restricted;
			
			\item for every $i\in[m]$, $\abs{B_i}\le\eta\abs{A_i}$ and $B_i$ is $\theta$-tight to $A_i$; and
			
			\item if $t>0$, then for every $i\in[t]$,
			$\abs{D_i}>2\eta^{-1}\abs{L_i}\ge\eta^{-1}\abs{L_i}$,
			$D_i$ is $\vep$-restricted (since it is $\vep_t$-restricted),
			and $L_i$ is $\theta$-tight to $D_i$.
		\end{itemize}
		By renumbering the above sets if necessary, we may assume there exist $q\in\{0,1,\ldots,m\}$ and $r\in\{0,1,\ldots,t\}$
		such that $B_i\ne\emptyset$ for all $i\in[q]$
		and $B_i=\emptyset$ for all $i\in[m]\setminus[q]$,
		and $L_i\ne\emptyset$ for all $i\in[r]$
		and $L_i=\emptyset$ for all $i\in[t]\setminus[r]$.
		Then the following partition of $V(G)$
		\[A_1,\ldots,A_q,D_1,\ldots,D_r;\,
		B_1,\ldots,B_q,L_1,\ldots,L_r;\,
		A_{q+1},\ldots,A_m,D_{r+1},\ldots,D_t,C_1,\ldots,C_n\]
		has the desired property, because
		$q+r\le m+t\le {t\choose2}+t={t+1\choose2}\le{h\choose2}$,
		and 
		\[(m-q)+(t-r)+n
		\le {t\choose2}+t+t\cdot \phi(\delta',\eta')
		\le {h\choose2}+(h-1)\cdot\phi(\delta',\eta')
		=N.\]
		This proves \cref{lem:main}.
	\end{proof}
	We remark that, according to~\cite[Theorem 1.4]{MR1887082}, the choice $\gamma_{\ref{lem:full}}(c,\vep)=\frac12(2\vep)^{12/c}$ is optimal up to a constant factor in the exponent.
	This choice leads to bounds on $\kappa_{\ref{lem:main}}^{-1}(H,\vep,\eta,\theta)$ and $N_{\ref{lem:main}}(H,\vep,\eta,\theta)$
	which are towers of twos of height depending on $h$ with $(\vep\eta\theta)^{-1}$ on top,
	which results in the tower-type dependence of $\kappa_{\ref{thm:copies}}^{-1}(H,\vep)$ and $N_{\ref{thm:copies}}(H,\vep)$ on $\vep^{-1}$ and $h$ mentioned in the introduction.
	
	\section{Finishing the proof}
	With \cref{lem:main} in hand,
	it now suffices to make obvious changes to~\cite[Section 3]{chudnovsky2021-1} to complete the proof of \cref{thm:copies}.
	For $\vep>0$, an integer $k\ge0$, and a graph $G$,
	a {\em $(k,\vep)$-path-partition} of $G$ is a sequence $(W_0,W_1,\ldots,W_k)$ of disjoint nonempty sets with union $V(G)$ such that for every $i\in\{0,1,\ldots,k-1\}$,
	\begin{itemize}
		\item $W_i$ is $\vep$-restricted in $G$;
		
		\item $\abs{W_i}\ge 12\abs{W_k}$; and
		
		\item $W_{i+1}\cup\cdots\cup W_k$ is $\frac{1}{12}\vep$-tight to $W_i$.
	\end{itemize}
	
	We need the following result from~\cite[Theorem 3.3]{chudnovsky2021-1}.
	\begin{lemma}
		[\cite{chudnovsky2021-1}]
		\label{lem:partition}
		For all $\vep\in(0,\frac{1}{3})$,
		every graph with a $(\ceil{4\vep^{-1}},\frac{1}{4}\vep)$-path-partition is $(2400\vep^{-2},\vep)$-restricted.
	\end{lemma}
	The following lemma allows us to \dd lengthen\ee{} a given path-partition of length less than $\ceil{4\vep^{-1}}$ by one,
	at the cost of removing a small number of vertices.
	\begin{lemma}
		\label{lem:key}
		Let $\vep\in(0,\frac{1}{3})$ and $K:=\ceil{4\vep^{-1}}$.
		Let $H$ be a graph with $h:=\abs H\ge2$.
		Let
		\begin{gather*}
			\vep':=h^{-2K}\vep,
			\qquad\eta:=h^{-2},
			\qquad\theta:=\frac{1}{12}h^{-2K}\vep,\\
			\kappa=\kappa_{\ref{lem:key}}(H,\vep)
			:=h^{-2Kh}\cdot\kappa_{\ref{lem:main}}(H,\vep',\eta,\theta),
			\qquad
			N=N_{\ref{lem:key}}(H,\vep):=N_{\ref{lem:main}}(H,\vep',\eta,\theta).
		\end{gather*} 
		Let $k$ be an integer with $0\le k\le K$.
		Let $d\ge0$, and let $G$ be a graph with $\ind(G)\le \kappa d^h$
		such that $G$ has a  $(k,h^{2(k-K)}\vep)$-path-partition $(W_0,W_1,\ldots,W_k)$.
		Then there is a set $S\subset V(G)$ with $\abs{S}\le h^{-2k}d$ such that $G\setminus S$ is $(h^{2(K-k)}(2400\vep^{-2}+N)-N,\vep)$-restricted.
	\end{lemma}
	\begin{proof}
		We proceed by backward induction on $k$.
		If $k=K$ then the conclusion follows by~\cref{lem:partition}.
		We may assume that $k<K$ and that the lemma holds for $k+1$.
		Since 
		\[\ind(G[W_k])\le\kappa d^h
		\le \kappa_{\ref{lem:main}}(H,\vep',\eta,\theta)\cdot (h^{-2(k+1)}d)^h,\]
		by~\cref{lem:main} applied to $G[W_k]$ with $d$ replaced by $h^{-2(k+1)}d$,
		there is $T\subset W_k$ with $\abs T\le h^{-2(k+1)}d$ such that $W_k\setminus T$ can be partitioned into nonempty sets
		\[A_1,\ldots,A_m,B_1,\ldots,B_m,C_1,\ldots,C_n\]
		where $m\le {h\choose2}$ and $n\le N$, such that
		\begin{itemize}
			\item $A_1,\ldots,A_m,C_1,\ldots,C_n$ are $\vep'$-restricted in $G$; and
			
			\item for every $i\in[m]$, $\abs{B_i}\le\eta\abs{A_i}$ and $B_i$ is $\theta$-tight to $A_i$.
		\end{itemize}
		
		If $m=0$ then $G\setminus T$ is $(k+N,\vep)$-restricted and we are done
		(note that $h\ge2$ and $k\le 4\vep^{-2}$);
		thus we may assume $m\ge1$.
		It follows that $\abs{W_k}\ge\abs{A_1}\ge\eta^{-1}\abs{B_1}\ge h^2\ge2m$,
		and so $\abs{W_i}\ge 12\abs{W_k}\ge 24m$ for all $i\in\{0,1,\ldots,k-1\}$.
		Thus for each such $i$, $W_i$ has a partition $W_i^1\cup\cdots\cup W_i^m$ with
		$\abs{W_i^j}\ge\floor{\frac{1}{m}\abs{W_i}}\ge\frac{1}{2m}\abs{W_i}\ge h^{-2}\abs{W_i}$ for all $j\in[m]$.
		Let $U_j:=\bigcup_{i=0}^{k-1}W_i^j\cup(A_j\cup B_j)$ for every $j\in[m]$.
		\begin{claim}
			\label{claim:key}
			For all $j\in[m]$, $(W_0^j,W_1^j,\ldots,W_{k-1}^j,A_j,B_j)$ is a $(k+1,h^{2(k+1-K)}\vep)$-path-partition of $G[U_j]$.
		\end{claim}
		\begin{subproof}
			It suffices to observe the following.
			\begin{itemize}
				\item $A_j$ is $h^{2(k+1-K)}\vep$-restricted
				since it is $\vep'$-restricted and $\vep'=h^{-2K}\vep$;
				and also, for each $i\in\{0,1,\ldots,k-1\}$,
				$W_i^j$ is $h^{2(k+1-K)}\vep$-restricted
				since $W_i$ is $h^{2(k-K)}\vep$-restricted and $\abs{W_i^j}\ge h^{-2}\abs{W_i}$.
				
				\item For every $i\in\{0,1,\ldots,k-1\}$, since $12\abs{W_k}\le\abs{W_i}\le h^2\abs{W_i^j}$, we have 
				\[12\abs{B_j}\le 12\eta\abs{A_j}
				\le\min(12h^{-2}\abs{W_k},\abs{A_j})
				\le \min(\abs{W_i^j},\abs{A_j}).\]
				
				\item $B_j$ is $\frac{1}{12}h^{2(k+1-K)}\vep$-tight to $A_j$ by the definition of $\theta$;
				and also,
				for every $i\in\{0,1,\ldots,k-1\}$,
				$(W_{i+1}^j\cup\cdots\cup W_{k-1}^j)\cup(A_j\cup B_j)$ is $\frac{1}{12}h^{2(k+1-K)}\vep$-tight to $W_i^j$
				since $W_{i+1}\cup\cdots\cup W_k$ is $\frac{1}{12}h^{2(k-K)}\vep$-tight to $W_i$
				and $\abs{W_i^j}\ge h^{-2}\abs{W_i}$.
			\end{itemize}
			
			This proves \cref{claim:key}.
		\end{subproof}
		By \cref{claim:key} and induction, for each $j\in[m]$, there is a set $S_j\subset U_j$ with $\abs{S_j}\le h^{-2(k+1)}d$ such that
		$G[U_j\setminus S_j]$ is $(h^{2(K-k-1)}(2400\vep^{-2}+N)-N,\vep)$-restricted.
		Put $S:=\bigcup_{j\in[m]}S_j\cup T$;
		then $\abs{S}\le (m+1)h^{-2(k+1)}d\le h^{-2k}d$ as $h\ge2$, and since
		\begin{align*}
			m\cdot (h^{2(K-k-1)}(2400\vep^{-2}+N)-N)+n
			&\le h^2\cdot(h^{2(K-k-1)}(2400\vep^{-2}+N)-N)+N\\
			&\le h^{2(K-k)}(2400\vep^{-2}+N)-N,
		\end{align*}
		we see that $G\setminus S$ is $(h^{2(K-k)}(2400\vep^{-2}+N)-N,\vep)$-restricted.
		This proves \cref{lem:key}.
	\end{proof}
	We are now ready to finish the proof of \cref{thm:copies}, which we restate here for the reader's convenience.
	\begin{theorem}
		\label{thm:main}
		For every $\vep>0$ and every graph $H$, there exist $\kappa=\kappa(H,\vep)>0$ and $N=N(H,\vep)>0$
		such that for every $d\ge0$ and every graph $G$ with $\ind(G)\le\kappa d^h$,
		there is a set $S\subset V(G)$ with $\abs{S}\le d$ such that $G\setminus S$ is $(N,\vep)$-restricted.
	\end{theorem}
	\begin{proof}
		We may assume $h\ge2$.
		Let $\kappa:=\kappa_{\ref{lem:key}}(H,\vep)$
		and $N:=h^{2K}(2400\vep^{-2}+N_{\ref{lem:key}}(H,\vep))$.
		Then \cref{thm:main} follows from \cref{lem:key} applied to the $(0,h^{-2K}\vep)$-partition $V(G)$ of $G$.
	\end{proof}
	\hypersetup{bookmarksdepth=-1}
	\subsection*{Acknowledgements}
	We would like to thank Alex Scott and Paul Seymour for stimulating discussions.
	We are also grateful to the anonymous referees for helpful comments, in particular to a referee who informed us about~\cite{MR1887082}.
	\hypersetup{bookmarksdepth=1}
	\bibliographystyle{abbrv}

\end{document}